\documentclass[sn-mathphys-num]{sn-jnl}

\usepackage{graphicx}%
\usepackage{multirow}%
\usepackage{amsmath,amssymb,amsfonts}%
\usepackage{amsthm}%
\usepackage{mathrsfs}%
\usepackage[title]{appendix}%
\usepackage{xcolor}%
\usepackage{textcomp}%
\usepackage{manyfoot}%
\usepackage{booktabs}%
\usepackage{algorithm}%
\usepackage{algorithmicx}%
\usepackage{algpseudocode}%
\usepackage{listings}%
\usepackage{tikz}
\usetikzlibrary{fit,patterns,decorations.pathmorphing,decorations.pathreplacing,calc}

\theoremstyle{thmstyleone}%
\newtheorem{theorem}{Theorem}%
\newtheorem{proposition}[theorem]{Proposition}%

\theoremstyle{thmstyletwo}%
\newtheorem{corollary}[theorem]{Corollary}%
\newtheorem{lemma}[theorem]{Lemma}%
\newtheorem{example}[theorem]{Example}%
\newtheorem{remark}[theorem]{Remark}%

\newtheorem{algo}[theorem]{Algorithm}
\newtheorem{conjecture}[theorem]{Conjecture}

\theoremstyle{thmstylethree}%
\newtheorem{definition}{Definition}%

\newcommand{\N}{\mathbb{N}}
\newcommand{\cA}{\mathcal{A}}
\newcommand{\cS}{\mathcal{S}}
\newcommand{\infw}[1]{\mathbf{#1}}
\newcommand{\word}[1]{u_{#1}}
\newcommand{\length}[1]{U_{#1}}
\newcommand{\eps}{\varepsilon}

\newcommand{\lowerb}[1]{{P_{#1}}}
\newcommand{\upperw}[1]{{q_{#1}}}
\newcommand{\upperb}[1]{{Q_{#1}}}
\newcommand{\invlex}[1]{\mathbin{{#1}_{-}}}
\DeclareMathOperator{\rep}{rep}
\DeclareMathOperator{\val}{val}
\DeclareMathOperator{\card}{Card}

\def\orcidID#1{\unskip$^{[#1]}$}

\begin{document}

\title[String attractors of some simple-Parry automatic sequences]{String attractors of some simple-Parry automatic sequences}

\author*[1]{\fnm{France} \sur{Gheeraert}\orcidID{0000-0002-8970-173X}}\email{france.gheeraert@ru.nl}

\author[2]{\fnm{Giuseppe} \sur{Romana}\orcidID{0000-0002-3489-0684}}\email{giuseppe.romana01@unipa.it}

\author[3]{\fnm{Manon} \sur{Stipulanti}\orcidID{0000-0002-2805-2465}}\email{m.stipulanti@uliege.be}

\affil*[1]{\orgdiv{Department of Mathematics}, \orgname{Radboud University}, \orgaddress{\street{Heyendaalseweg 135}, \city{Nijmegen}, \postcode{6525AJ}, \country{Netherlands}}}

\affil[2]{\orgdiv{Dipartimento di Matematica e Informatica}, \orgname{Università di Palermo}, \orgaddress{\street{Via Archirafi 34}, \city{Palermo}, \postcode{90123}, \country{Italy}}}

\affil[3]{\orgdiv{Department of Mathematics}, \orgname{University of Li\`ege}, \orgaddress{\street{All\'ee de la D\'ecouverte 12}, \city{Li\`ege}, \postcode{4000}, \country{Belgium}}}

\abstract{Firstly studied by Kempa and Prezza in 2018 as the cement of text compression algorithms, string attractors have become a compelling object of theoretical research within the community of combinatorics on words.
In this context, they have been studied for several families of finite and infinite words.
In this paper, we obtain string attractors of prefixes of particular infinite words generalizing $k$-bonacci words (including the famous Fibonacci word) and related to simple Parry numbers.
In fact, our description involves the numeration systems classically derived from the considered morphisms.
This extends our previous work published in the international conference WORDS 2023.
}

\keywords{String attractors, Numeration systems, Parry numbers, Automatic sequences, Morphic sequences, Fibonacci word}

\pacs[MSC Classification]{Primary: 68R15. Secondary: 05A05,  11A67, 68P05, 68Q45.}

\maketitle

\section{Introduction}\label{sec: intro}

Introduced in the field of data compression by Kempa and Prezza~\cite{Kampa-Prezza} in 2018, the concept of \emph{string attractor} can be conceptualized as follows: within a finite word, it is a set of positions that enables to catch all distinct factors.
Since then, questions related to string attractors have drawn the attention of many researchers from various scientific fields.
From the point of view of the theory of algorithmic complexity, the problem of finding a smallest string attractor is NP-hard~\cite{Kampa-Prezza}.
In parallel, string attractors also have applications in combinatorial pattern matching~\cite{ChristiansenEKN21}.
In order to understand the best way to measure data compressibility by exploiting repetitiveness in strings, measures have recently been introduced in relation to string attractors~\cite{KNP2020}.

Quickly after, combinatorics-on-words researchers have quite naturally seized the notion and made a systematic topic of research out of it.
String attractors have been studied in the context of some (families of) sequences:  automatic sequences~\cite{Schaeffer-Shallit-2020} with some focus on the ubiquitous Thue--Morse word~\cite{TM-sa,Schaeffer-Shallit-2020,Dolce2023} and the period-doubling word~\cite{Schaeffer-Shallit-2020}, the famous Sturmian words~\cite{saandcow,SA-LATIN22,Dvorakova2022} and their extension known as episturmian words~\cite{Dvorakova2022}, the Tribonacci word~\cite{Schaeffer-Shallit-2020} and more generally the $k$-bonacci words~\cite{CGRRSS23}, some binary generalized pseudostandard sequences~\cite{DH2023}, and bi-infinite words~\cite{BGHdM}.
Besides analyzing important families of words, another classical topic in combinatorics on words is the study of complexity functions for infinite words, such as the distinguished factor complexity function and the so-called abelian~\cite{CH1973} and binomial~\cite{RS2015} complexity functions.
A first complexity function based on string attractors was introduced in~\cite{Schaeffer-Shallit-2020} and considered for automatic sequences and linearly recurrent infinite words.
In addition to further studying such a complexity function, the authors of~\cite{SA-LATIN22,CGRRSS23} introduce and examine two other string attractor-based complexity functions.

Historically, the bond between string attractors and numeration systems was observed for the first time in~\cite{CGRRSS23}.
There, the authors consider generalizations of the Fibonacci word to larger alphabets (on $k$ letters, the corresponding word is called the \emph{$k$-bonacci word}) and show that some of their strings attractors rely on the well-known $k$-bonacci numbers.
These infinite words are \emph{morphic}, i.e., they are obtained as images, under  codings, of fixed points of morphisms.
Consequently, we currently believe that the link between string attractors and numeration systems can be adapted to other morphic sequences and we therefore raise the following general question:

\smallskip

\noindent\textbf{Question.} Given a morphic sequence $\infw{z}$, does there exist a numeration system $\cS$ such that $\infw{z}$ is $\cS$-automatic and (minimal) string attractors of the prefixes of $\infw{z}$ are easily described using $\cS$?

\smallskip

In this paper, to support this question, we study a particular family of morphic words.
More precisely, given parameters in the shape of a length-$k$ word $c=c_0 \cdots c_{k-1} \in \N^k$, we define the morphism $\mu_c$ such that $\mu_c(i) = 0^{c_i} \cdot (i+1)$ for all $i\in\{0,\ldots,k-2\}$ and $\mu_c(k-1) = 0^{c_{k-1}}$.
When it exists, we then look at the fixed point of $\mu_c$.
From a combinatorics-on-words point of view, this family is interesting as it generalizes the $k$-bonacci morphisms.
However, their main interest stems from numeration systems.
Given a simple Parry number $\beta$ (a real number for which the $\beta$-expansion of $1$ is of the form $u0^\omega$ for some finite word $u$), Fabre~\cite{Fabre} associated with it a morphism $\mu$ which is strongly related to the automaton corresponding to the base-$\beta$ numeration system (for real numbers), or alternatively to the greedy numeration system associated with the sequence $(|\mu^n(0)|)_{n \in \N}$ (for integers)~\cite{Bertrand89}. This morphism can also be used to understand the base-$\beta$ integers~\cite{Fabre,Thurston}.
It turns out that $\mu$ corresponds to $\mu_c$ where $c$ is the $\beta$-representation of 1. Note that, due to a result of Parry~\cite{Parry}, not every $c$ corresponds to a simple Parry number.
Combinatorially speaking, the fixed points of these particular morphisms have been studied for instance in~\cite{BFGK98,FMP2004,AMPF2006,T2015}.

Furthermore, the techniques used so far to obtain string attractors of infinite words do not apply to the words we consider.
Indeed, on the one hand, they are not necessarily episturmian, so we cannot use the approach from~\cite{Dvorakova2022}.
On the other hand, for some parameter $c\in \N^k$, the corresponding numeration system is not \emph{addable}, meaning that the addition within the numeration system is not recognizable by a finite automaton.
For example, this is the case of $c=3203$~\cite{Frougny1997}.
As a consequence, we cannot follow the methods from~\cite{Schaeffer-Shallit-2020}; in particular, we study words outside the framework needed to use the software \texttt{Walnut}~\cite{Shallit2022}.

Under some conditions on the parameters, we show that the prefixes of the fixed point admit string attractors strongly related to the associated numeration system. Moreover, they are of size at most one more than the alphabet size, and are therefore nearly optimal. Finally, we provide an infinite family of words for which these string attractors have minimal size.

This work extends and completes some of the results presented in our first exploratory paper on this topic~\cite{GRS-WORDS2023}. In this version, we emphasize the link with numeration systems and simple Parry numbers. We also provide new results, and as a consequence, have restructured and simplified most of the proofs.

This paper is organized as follows.
In Section~\ref{sec: prelim}, we first recall some classical notions of combinatorics on words.
In particular, we introduce the family of words that we will focus on.
We then show how these words are related to numeration systems in Section~\ref{sec: ANS}, and we study some of their properties such as greediness.
Section~\ref{sec: SA} is devoted to the construction of string attractors of the words.
More precisely, we give the precise conditions under which every prefix admits a string attractor included in a particular subset, and we exhibit such a string attractor when the conditions are met.
Then in Section~\ref{sec: frac power and anti-Lyndon}, we analyze these conditions and give alternative formulations in terms of the numeration system and of the parameters $c_0,\dots,c_{k-1}$.
Finally, in Section~\ref{sec: optimality}, we discuss the optimality of the string attractors that we obtained and describe an infinite family of words for which they are optimal.
\section{Preliminaries}\label{sec: prelim}

\subsection{Words}

We start with the bare minimum on words and introduce some notations.

Let $A$ be an alphabet either finite or infinite (for instance, we will consider words over the set of non-negative integers $\N$).
The length of a word is its number of letters and will be denoted with vertical bars $|\cdot|$.
We let $\eps$ denote the empty word, and $A^*$ denote the set of finite words over $A$.
For any integer $n\geq0$, we let $A^n$ be the set of length-$n$ words over $A$.
If $w = xyz$ for some $x,y,z \in A^*$, then $x$ is a \emph{prefix}, $y$ is a \emph{factor}, and $z$ is a \emph{suffix} of $w$.
A factor of a word is \emph{proper} if it is not equal to the initial word.
A word $v$ is a \emph{fractional power} of a non-empty word $w$ if there exist $\ell \geq 0$ and $x$ a prefix of $w$ such that $v = w^\ell x$.
Infinite words are written in bold and we start indexing them at $0$.
We use classical notations of intervals to denote portions of words.
For a non-empty word $u\in A^*$, we let $u^\omega$ denote the concatenation of infinitely many copies of $u$, that is, $u^\omega= u u u \cdots$.

Let $\leq$ be a total order on $A$. The \emph{lexicographic order} on $A^*$ induced by $\leq$ is defined as follows: for $x,y\in A^*$, we say that $x$ is \emph{lexicographically smaller than} $y$, and we write $x < y$, if either $x$ is a proper prefix of $y$, or $x=zax'$ and $y=zby'$ for some letters $a,b$ with $a < b$.
We write $x\le y$ if $x$ is lexicographically smaller than or equal to $y$.
The \emph{genealogical order}, also known as \emph{radix order}, on $A^*$ induced by $\leq$ is defined as follows: for $x,y\in A^*$, we say that $x$ is \emph{genealogically smaller than} $y$, and we write $x<_{\text{gen}} y$, if either $|x|<|y|$, or $|x|=|y|$ and $x=zax'$ and $y=zby'$ for some letters $a,b$ with $a< b$.
We write again $x\le_{\text{gen}} y$ if $x$ is genealogically smaller than or equal to $y$.

A non-empty word $w\in A^*$ is \emph{primitive} if $w=u^n$ for $u\in A^*\setminus\{\eps\}$ implies $n=1$.
Two words are \emph{conjugates} if they are cyclic permutation of each other.

A word is \emph{Lyndon} if it is primitive and lexicographically minimal among its conjugates for some given order.
Defined in the 50's, Lyndon words are not only classical in combinatorics on words but also of utmost importance.
See~\cite{Lothaire97} for a presentation.
A celebrated result in combinatorics on words is that London words form a so-called \emph{complete factorization of the free monoid}.

\begin{theorem}[Chen-Fox-Lyndon~\cite{CFL58}]\label{thm:CFL-fact}
For every non-empty word $w\in A^*$, there exists a unique factorization $(\ell_1,\ldots, \ell_n)$ of $w$ into Lyndon words over $A$ such that $\ell_1 \geq\ell_2 \geq\cdots \geq\ell_n$.
\end{theorem}

Several variations of Lyndon words have been considered lately: generalized Lyndon~\cite{Reut}, anti-Lyndon~\cite{Gewurz-Merola-2012}, inverse Lyndon~\cite{BdFZZ2018}, and Nyldon~\cite{Nyldon}.
In this text, we will use the second.

\begin{definition}
Let $(A,\leq)$ be a totally ordered alphabet.
We let $\leq_-$ denote the \emph{inverse order} on $A$, i.e., $b <_- a$ if and only if $a < b$ for all $a,b\in A$.
We also let $\invlex{\leq}$ denote the \emph{inverse lexicographic order} which is the lexicographic order induced by $\leq_-$.
A word is \emph{anti-Lyndon} if it is Lyndon with respect to the inverse lexicographic order.
\end{definition}

Otherwise stated, a word is anti-Lyndon if it is primitive and lexicographically maximal among its conjugates.

\begin{example}
\label{Ex: anti-Lyndon}
Let $A=\{0,1\}$ with $0<1$, so $1\invlex{<}0$. The first few anti-Lyndon words, ordered by length, are $1$, $0$, $10$, $110$, $100$, $1110$, $1100$, and $1000$.
\end{example}

\subsection{Morphisms and fixed points of interest}\label{sec: morphism and fp}

A \emph{morphism} is a map $f \colon A^* \to B^*$, where $A,B$ are alphabets, such that $f(xy)=f(x)f(y)$ for all $x,y\in A^*$.
The morphism $f$ is \emph{prolongable} on the letter $a\in A$ if $f(a)=ax$ for some $x\in A^*$ and $f^n(x)\neq \eps$ for all $n\geq0$.
In this section, we consider a specific family of morphisms defined as follows.
Note that they appear under the name \emph{generic $k$-bonacci} morphisms in~\cite[Example 2.11]{Rigo2014-vol2}.

\begin{definition}\label{def:mu-c}
Let $k \geq 2$ be an integer and let $c_0, \ldots, c_{k-1} \geq 0$ be $k$ parameters often summarized in the shape of a word $c = c_0 \cdots  c_{k-1} \geq 0^k$.
The morphism $\mu_c \colon \{0, \ldots, k-1\}^* \to \{0, \ldots, k-1\}^*$ is given by $    \mu_c(i) = 0^{c_i} \cdot (i+1)$ for all $i\in\{0,\ldots,k-2\}$ and $\mu_c(k-1) = 0^{c_{k-1}}$.
For all $n\geq0$, we then define $\word{c, n} = \mu_c^n(0)$ and $\length{c, n} = |\word{c, n}|$.
\end{definition}

When the context is clear, we will usually omit the subscript $c$ in Definition~\ref{def:mu-c}.

\begin{example}
\label{ex:c102 - words and lengths}
When $c=1^k$, we recover the $k$-bonacci morphism and words.
For $k=3$ and $c = 102$, the first few iterations of the corresponding morphism $\mu_c \colon 0 \mapsto 01, 1\mapsto 2, 2 \mapsto 00$ are given in Table~\ref{tab:c102 - words and lengths}.
Some specific factorization of the words  $(\word{c, n})_{n \geq 0}$ is highlighted in Table~\ref{tab:c102 - words and lengths}.

\begin{table}
    \caption{Construction of the sequences $(\word{n})_{n \geq 0}$ and $(\length{n})_{n \geq 0}$ for $c = 102$.}
    \label{tab:c102 - words and lengths}
    \begin{tabular}{c||c|c|c|c|c|c}
        $n$ & $0$ & $1$ & $2$ & $3$ & $4$ & $5$\\[1pt]
        \hline\rule{0pt}{.8\normalbaselineskip}
        \!\!$\word{n}$ & $0$ & $01$ & $012$ & $01200$ & $012000101$ & $012000101012012$ \\
        fact. of $\word{n}$ & $0$ & $\word{0}^1 \cdot 1$ & $\word{1}^1 \word{0}^0 \cdot 2$ & $\word{2}^1 \word{1}^0 \word{0}^2$ & $\word{3}^1 \word{2}^0 \word{1}^2$ & $\word{4}^1 \word{3}^0 \word{2}^2$ \\
        $\length{n}$ & $1$ & $2$ & $3$ & $5$ & $9$ & $15$
    \end{tabular}
\end{table}
\end{example}

The factorization presented in the previous example can be stated in general.
It gives a recursive definition of the words $(\word{c, n})_{n\geq0}$ and can be proven using a simple induction.

\begin{proposition}
\label{P:alternative definition of the words}
For all $c = c_0 \cdots  c_{k-1} \geq 0^k$, we have
\[
    \word{n} =
    \begin{cases}
        \left(\displaystyle\prod_{i = 0}^{n-1} \word{n-i-1}^{c_i}\right) \cdot n, & \text{if } 0 \leq n \leq k-1;\\
        \displaystyle\prod_{i = 0}^{k-1} \word{n-i-1}^{c_i}, & \text{if } n \geq k.
    \end{cases}
\]
\end{proposition}

As a consequence of Proposition~\ref{P:alternative definition of the words}, the sequence $(\length{n})_{n \geq 0}$ respects the following recurrence relation:
\begin{equation}
\label{eq: recurrence for U_n}
    \length{n}
=
\begin{cases}
    1 + \sum_{i = 0}^{n-1} c_i \length{n-i-1}, & \text{if } 0\le n \leq k-1; \\
    \sum_{i = 0}^{k-1} c_i \length{n-i-1}, & \text{if } n \geq k.
\end{cases}
\end{equation}

In the rest of the paper, we will assume the following working hypothesis \eqref{eq:WH} on $c$:
\begin{equation}\tag{WH}
\label{eq:WH}
    c = c_0 \cdots c_{k-1} \geq 0^k \text{ with } c_0, c_{k-1} \geq 1.
\end{equation}
The condition $c_{k-1} \geq 1$ ensures both that the recurrence relation is of order $k$ and that the morphism $\mu_c$ is non-erasing, which is a classical assumption in combinatorics on words.
Moreover, the condition $c_0 \geq 1$ guarantees that $\mu_c$ is prolongable. Under~\eqref{eq:WH}, the morphism $\mu_c$ has an infinite fixed point starting with 0 denoted $\infw{u} := \lim_{n \to \infty} \word{n}$.

We make the following combinatorial observation.

\begin{remark}
\label{R:extensions of letters}
Under~\eqref{eq:WH}, using Proposition~\ref{P:alternative definition of the words}, a simple induction shows that the letter $i\in \{1,\ldots,k-1\}$ can only be followed by $0$ and/or $i+1$ (and only $0$ in the case $i=k-1$) in $\infw{u}$ and in $u_n^\omega$.
\end{remark}

\section{Link with numeration systems}\label{sec: ANS}

In this section, specific definitions will be recalled.
For the reader unfamiliar with the theory of numeration systems, we refer to~\cite[Chapters 2 and 3]{CANT2010} for an introduction and some advanced concepts.

A \emph{numeration system} (for natural numbers) can be defined as a triple $\cS = (A, \rep_\cS, L)$, where $A$ is an alphabet and $\rep_\cS \colon \N \to A^*$ is an injective function such that $L = \rep_\cS(\N)$.
The map $\rep_{\cS}$ is called the \emph{representation function} and $L$ is the \emph{numeration language}.
If $\rep_\cS(n) = w$ for some integer $n\geq 0$ and some word $w\in A^*$, we say that $w$ is the \emph{representation (in $\cS$)} of $n$ and we define the \emph{valuation (in $\cS$)} of $w$ by $\val_\cS(w) = n$.
Note that, when the context is clear, we omit the subscript $\cS$ in $\rep$ and $\val$.

Any given prolongable morphism naturally gives rise to a numeration system that we will call the \emph{associated Dumont-Thomas numeration system}~\cite{Dumont-Thomas-1989}. These are based on particular factorizations of the prefixes of the fixed point.
We only give here the definition in the particular case of the morphisms studied in this paper but the interested reader can find the general case in the original paper~\cite{Dumont-Thomas-1989}.
\begin{proposition}[Dumont-Thomas~\cite{Dumont-Thomas-1989}]
\label{P:Dumont-Thomas}
Let $c$ satisfy~\eqref{eq:WH}.
For all $n \geq 1$, there exist unique integers $N, \ell_0, \ldots, \ell_N\geq 0$ such that $\ell_0 \geq 1$,
$\infw{u}[0,n) = \word{N}^{\ell_0} \cdots \word{0}^{\ell_N}$, and this factorization verifies the following: $\word{N+1}$ is not a prefix of $\infw{u}[0,n)$ and, for all $0 \leq i \leq N$, $\word{N}^{\ell_0} \cdots \word{N-i+1}^{\ell_{i-1}} \word{N-i}^{\ell_i + 1}$ is not a prefix of $\infw{u}[0,n)$.
\end{proposition}

Recall that a numeration system based on a suitable sequence of integers $(U_n)_{n\geq0}$ is called \emph{greedy} when, at each step of the decomposition of any integer, the largest possible term of the sequence $(U_n)_{n\geq0}$ is chosen; formally, we use the Euclidean algorithm in a greedy way.
As the conditions on the factorization in Proposition~\ref{P:Dumont-Thomas} resemble that of greedy representations in numeration systems, we will refer to it as being \emph{word-greedy}.

For a given $c$ satisfying~\eqref{eq:WH}, we then let $\cS_c$ denote the numeration system associated with the representation function $\rep_{\cS_c} \colon \N \to \N^*$ mapping $n$ to $\rep_{\cS_c}(n) = \ell_0 \cdots \ell_N$,
where the integers $\ell_0, \ldots, \ell_N$ verify the conditions of Proposition~\ref{P:Dumont-Thomas} for $n$. By convention, we set $\rep_{\cS_c}(0) = \eps$.

\begin{example}
\label{Ex:c102 - numeration system}
Using Example~\ref{ex:c102 - words and lengths} for $c = 102$, the representations of the first few integers are given in Table~\ref{tab:c102 - representations}.
The word-greedy factorization of each prefix is highlighted in the second row, leading to the representation of the corresponding integer in the third row.
\begin{table}
\caption{Illustration of the numeration system $\cS_c$ for $c = 102$.}
    \label{tab:c102 - representations}
    \begin{tabular}{c||c|c|c|c|c|c|c|c|c}
     $n$ & $0$ & $1$ & $2$ & $3$ & $4$ & $5$ & $6$ & $7$ & $8$\\[1pt]
     \hline
     \rule{0pt}{.8\normalbaselineskip}
     \!\!$\infw{u}[0,n)$ & $\eps$ & $0$ & $01$ & $012$ & $012 \cdot 0$ & $01200$ & $01200 \cdot 0$ & $01200 \cdot 01$ & $01200 \cdot 01 \cdot 0$\\
     $\rep_{\cS_c}(n)$ & $\eps$ & $1$ & $10$ & $100$ & $101$ & $1000$ & $1001$ & $1010$ & $1011$
    \end{tabular}
\end{table}
\end{example}

\begin{remark}\label{R:valuation of numeration systems}
If $\rep_{\cS_c}(n)= \ell_0 \cdots \ell_N$, then $n
= |\word{N}^{\ell_0} \cdots \word{0}^{\ell_N}| 
= \sum_{i = 0}^N \ell_i \length{N-i}. $
In other words, $\val_{\cS_c}$ is given by the usual valuation function associated with the sequence $(\length{n})_{n \geq 0}$. Such a system is sometimes called a \emph{positional} numeration system. Note that this is not necessarily the case for the Dumont-Thomas numeration system associated with some other morphism.
\end{remark}

The Dumont-Thomas numeration systems are a particular case of abstract numeration systems introduced in~\cite{Lecomte-Rigo-2001}. A numeration system $\cS = (A, \rep, L)$ is said to be \emph{abstract} if $L$ is regular and $\rep(n)$ is the $(n+1)$st word of $L$ in the genealogical order.
We have the following result.

\begin{theorem}[{Rigo~\cite[Section 2.2]{Rigo2014-vol2}}]
\label{T:automata construction of numeration systems}
Let $\sigma \colon \{\alpha_0, \ldots, \alpha_d\}^* \to \{\alpha_0, \ldots, \alpha_d\}^*$ be a morphism prolongable on the letter $\alpha_0$.
We define the automaton $\cA_\sigma$ for which $\{\alpha_0, \ldots, \alpha_d\}$ is the set of states, $\alpha_0$ is the initial state, every state is final, and the (partial) transition function $\delta$ is such that, for each $\alpha \in \{\alpha_0, \ldots, \alpha_d\}$ and $0 \leq i \leq |\sigma(\alpha)| - 1$, $\delta(\alpha, i)$ is the $(i+1)$st letter of $\sigma(\alpha)$. If $\cS = (A, \rep, L)$ is the Dumont-Thomas numeration system associated with $\sigma$, then $L = L(\cA_\sigma) \setminus 0\N^*$ and $\rep(n)$ is the $(n+1)$st word of $L$ in the genealogical order.
\end{theorem}

\begin{example}
\label{Ex:c102 - automaton}
For $c = 102$, the automaton $\cA_{\mu_c}$ of Theorem~\ref{T:automata construction of numeration systems} is depicted in Figure~\ref{fig:c102 - automaton} (details are left to the reader). The first few accepted words (not starting with $0$) are, in genealogical order, $\eps$, $1$, $10$, $100$, $101$, $1000$, $1001$, $1010$, and $1011$, which indeed agree with the representations of the first few integers in Example~\ref{Ex:c102 - numeration system}.
\end{example}

\begin{figure}[h!tbp]
  \begin{center}
\begin{tikzpicture}[scale=.8,every node/.style={circle,minimum width=.8cm,inner sep=1.5pt}]
\node (0) at (0,0) [draw,circle] {$0$};
\node (1) at (3,0) [draw,circle] {$1$};
\node (2) at (6,0) [draw,circle] {$2$};
\draw [->] (0) edge[out=110,in=70,looseness=8] node[above=-.2] {$0$} (0);
\draw [->] (0) edge node[above=-.2,midway] {$1$} (1);
\draw [->] (1) edge node[above=-.2,midway] {0} (2);
\draw [->] (2) edge[out=225,in=-45] node[below=-.1] {0, 1} (0);
 \draw [->] (-1,0) -- (0);
\end{tikzpicture}
\caption{The automaton $\cA_{\mu_c}$ for $c = 102$.}
\label{fig:c102 - automaton}
\end{center}
\end{figure}

As the automaton in Theorem~\ref{T:automata construction of numeration systems} can be used to produce, for all $n\geq0$, the letter $\infw{u}_n$ when reading $\rep_{\cS_c}(n)$ by~\cite[Theorem 2.24]{Rigo2014-vol2}, we have the following.

\begin{corollary}\label{C:automatic sequences}
Let $c$ satisfy~\eqref{eq:WH}.
Then the sequence $\infw{u}$ is $\cS_c$-automatic.
\end{corollary}

Similarly to what is usually done in real base numeration systems, we will let  $\infw{d}^\star$ denote the periodization of $c$, that is, $\infw{d}^\star=(c_0 \cdots c_{k-2} (c_{k-1}-1))^\omega$.
Using Theorem~\ref{T:automata construction of numeration systems}, we deduce the next result.

\begin{lemma}
\label{L:decomposition of T_n+1 - 1}
Under~\eqref{eq:WH}, for all $n \geq 0$, we have $\rep_{\cS_c}(\length{n}) = 10^n$, the numbers having a representation of length $n+1$ are those in $[\length{n}, \length{n+1})$, and
$\rep_{\cS_c}(\length{n+1}-1) = \infw{d}^\star[0,n]$.
In particular, $\length{n+1}-1 = \sum_{i = 0}^n \infw{d}^\star_i \length{n-i}$.
\end{lemma}
\begin{proof}
The first claim directly follows by the definition of $\cS_c$, and the second one by the genealogical order.
The number $\length{n+1}-1$ is then represented by the maximal length-$(n+1)$ word accepted by the automaton $\cA_{\mu_c}$, which is the length-$(n+1)$ prefix of $\infw{d}^\star$.
\end{proof}

Note that, if the numeration system $\cS_c$ satisfies the greedy condition, this result follows from the characterization of numeration systems in terms of dynamical systems given by Bertrand-Mathis~\cite{Bertrand89,CCSDLT22}.
However, even though the function $\rep_{\cS_c}$ is obtained using the word-greedy factorization of prefixes of $\infw{u}$, the numeration system $\cS_c$ is not necessarily greedy as the following example shows.

\begin{example}
In Example~\ref{ex:c102 - words and lengths} for $c = 102$, we see that $\infw{u}[0,14) = 012000101 \cdot 012 \cdot 01$, so $\rep_{\cS_c}(14) = 10110$, while the greedy representation of $14$ associated with the sequence $(U_n)_{n \geq 0}$ is $11000$.
\end{example}

In fact, we have the following two characterizations.

\begin{lemma}
\label{L:greedyIFFlanguage}
Let $c$ satisfy~\eqref{eq:WH}. The numeration system $\cS_c = (A, \rep_{\cS_c}, L)$ is greedy if and only if,
for all $v \in L$ and for all $i \leq |v|$, the suffix of length $i$ of $v$ is smaller than or equal to $\infw{d}^\star[0,i)$.
Moreover, we then have
\[
    L = \{v=v_1 \cdots v_n \in \N^* \setminus 0 \N^* : \forall \, 1 \leq i \leq n, v_{n-i+1} \cdots v_n \leq \infw{d}^\star[0,i)\}.
\]
\end{lemma}
\begin{proof}
Let us denote $\cS = (A', \rep_\cS, L')$ the canonical greedy numeration system associated with the sequence $(U_n)_{n \geq 0}$. In particular, by uniqueness, $\cS_c$ is greedy if and only if $\cS_c = \cS$.
As $\cS_c$ is an abstract numeration system, $\rep_{\cS_c}$ respects the genealogical order, i.e., $n \leq m$ if and only if $\rep_{\cS_c}(n) \leq_{\text{gen}} \rep_{\cS_c}(m)$. So does $\rep_{\cS}$ by~\cite[Proposition 2.3.45]{CANT2010}. Hence, $\cS_c = \cS$ if and only if $L = L'$.
Moreover, for all $n \geq0$, $\rep_{\cS}(U_n) = 10^n$, so $L$ and $L'$ contain the same number of length-$n$ words by Lemma~\ref{L:decomposition of T_n+1 - 1}. Thus $L = L'$ if and only if $L \subseteq L'$.
By~\cite[Lemma 5.3]{Hollander}, we have
\[
    L' = \{v = v_1 \cdots v_n \in \N^* \setminus 0 \N^* : \forall \, 1 \leq i \leq n, v_{n-i+1} \cdots v_n \leq \rep_\cS(U_i - 1)\},
\]
so if $\cS_c$ is greedy, then $L = L'$ and, by Lemma~\ref{L:decomposition of T_n+1 - 1}, $\rep_\cS(U_i - 1) = \infw{d}^\star[0,i)$ so we conclude.
For the converse, let us proceed by contraposition and assume that $\cS_c$ is not greedy. Therefore, $L \not \subseteq L'$ and there exists $v \in L$ and $i \leq |v|$ such that $v_{n-i+1} \cdots v_n > \rep_\cS(U_i - 1)$. However, since $\infw{d}^\star[0,i)$ is also a representation of $U_i - 1$ associated with the sequence $(U_n)_{n \geq 0}$, we have $\rep_\cS(U_i - 1) \geq \infw{d}^\star[0,i)$ (see~\cite[Proposition 2.3.44]{CANT2010} for example). Therefore, the length-$i$ suffix of $v$ is strictly greater than $\infw{d}^\star[0,i)$, which ends the proof.
\end{proof}

\begin{theorem}
\label{T:greedy condition}
Let $c = c_0\cdots c_{k-1} \geq 0^k$ with $c_0, c_{k-1} \geq 1$.
The numeration system $\cS_c$ is greedy if and only if $c_0 \cdots c_{k-2}(c_{k-1} - 1)$ is lexicographically maximal among its conjugates.
\end{theorem}
\begin{proof}
Using Lemma~\ref{L:greedyIFFlanguage} and Theorem~\ref{T:automata construction of numeration systems}, $\cS_c$ is greedy if and only if, for all $n \geq 0$ and for all $0 \leq i \leq k-1$, any path $\ell_0 \cdots \ell_n$ starting in state $i$ in the automaton $\cA_{\mu_c}$ is such that $\ell_0 \cdots \ell_n \leq \infw{d}^\star[0,n]$.
However, by definition of $\cA_{\mu_c}$, the lexicographically biggest path of length $n$ starting in state $i$ is given by the prefix of length $n$ of $\left(c_i \cdots c_{k-2} (c_{k-1} - 1) c_0 \cdots c_{i-1}\right)^\omega$.
Therefore, we can conclude that $\cS_c$ is greedy if and only if $c_i \cdots c_{k-2} (c_{k-1} - 1) c_0 \cdots c_{i-1} \leq c_0 \cdots c_{k-2}(c_{k-1} - 1)$ for all $0 \leq i \leq k-1$, i.e., $c_0 \cdots c_{k-2}(c_{k-1} - 1)$ is maximal among its conjugates.
\end{proof}

\begin{example}
\label{Ex: two NS for Fibonacci}
Let $k=4$ and $c=1011$.
The sequence $\length{n}$ satisfies the recurrence relation $\length{n+4} = \length{n+3} + \length{n+1} + \length{n}$ with initial conditions $\length{0}=1$, $\length{1}=2$, $\length{2}=3$, and $\length{3}=5$. 
A simple induction shows that $(\length{n})_{n \geq 0}$ is in fact the sequence of Fibonacci numbers.
As $c_0 c_1 c_2(c_3 - 1) = 1010$ is maximal among its conjugates, the numeration system $\cS_c$ then corresponds to the classical Fibonacci numeration system, which can also be obtained with the parameter $c'=11$. In this case, $c_0 c_1 c_2(c_3 - 1) = 1010 = v^2$ with $v=10 = c'_0(c'_1 - 1)$, which is anti-Lyndon (see Example~\ref{Ex: anti-Lyndon}). 
\end{example}

The observation made in the previous example is more general, as we show below.
Recall that a real number $\beta$ is \emph{Parry} if the $\beta$-expansion $d_\beta(1)$ of $1$ is eventually periodic.
It is \emph{simple Parry} if $d_\beta(1)$ is finite, i.e., $d_\beta(1)$ is of the form $u0^\omega$ for some finite word $u$.
See~\cite[Chapter 2]{CANT2010} for more details.

\begin{proposition}
\label{pro: link with simple Parry}
    Let $c$ satisfy~\eqref{eq:WH} and be such that $c_0 \cdots c_{k-2}(c_{k-1} - 1)$ is maximal among its conjugates.
    \begin{enumerate}
        \item We have $c_0 \cdots c_{k-2}(c_{k-1} - 1) = (c'_0 \cdots c'_{j-2} (c'_{j-1} - 1))^\ell$ where $c'_0 \cdots c'_{j-2} (c'_{j-1} - 1)$ is an anti-Lyndon word.
        \item There exists a simple Parry number $\beta$ such that $d_\beta(1) = c'_0 \cdots c'_{j-1} 0^\omega$.
        \item We have $\cS_c = \cS_{c'}$.
        \item If $\sigma\colon i \mapsto i \bmod j$ for all $0 \leq i \leq k-1$, then $\infw{u}_{c'} = \sigma(\infw{u}_c)$.
    \end{enumerate}
\end{proposition}

\begin{proof}
The first item directly follows from the definition of anti-Lyndon words.

Let us turn to the proof of the second item.
As $c'_0 \cdots c'_{j-2}(c'_{j-1}-1)$ is anti-Lyndon by assumption, $c'_0 \cdots c'_{i-1} > c'_{j-i} \cdots c'_{j-2}(c'_{j-1}-1)$ for all $i\in\{1,\ldots,j-1\}$. Therefore, we obtain $c' > c'_0 \cdots c'_{i-1} \geq c'_{j-i} \cdots c'_{j-1}$. By a result of Parry~\cite[Corollary 4]{Parry}, this then implies that there exists some simple Parry number $\beta$ such that $c'0^\omega = d_\beta(1)$.

Let us turn to the proof of the third item.
Write $v=c'_0 \cdots c'_{j-2} (c'_{j-1} - 1)$ (simply put, to get $c'$, we add $1$ to the last letter of $v$).
By the first item, $c = v^{\ell-1}c'$ is a ``partial'' cyclization of $c'$.
In particular, by definition, we obtain $\infw{d}^\star_c = \infw{d}^\star_{v'}$ (where the dependence of $\infw{d}^\star$ on the chosen parameters is emphasized via a subscript).
The numeration systems $\cS_c$ and $\cS_{c'}$ thus coincide by Theorem~\ref{T:greedy condition} and Lemma~\ref{L:greedyIFFlanguage}.

Finally, let us show the fourth item. For all $i \not \equiv j-1 \bmod j$, we have
\[
\mu_{c'} \circ \sigma(i) = 0^{c'_{i \bmod j}} ((i \bmod j) + 1) = 0^{c_i} (i+1 \bmod j) = \sigma \circ \mu_c(i).
\]
Similarly, if $i \equiv j-1 \bmod j$, then
\begin{align*}
    \mu_{c'} \circ \sigma(i) = 0^{c'_{j-1}} 
    &=
    \begin{cases}
        0^{c_i + 1}, & \text{if } i \ne k-1; \\
        0^{c_{k-1}}, & \text{if } i = k-1;
    \end{cases} \\
    &= \sigma \circ \mu_c(i).
\end{align*}
This shows that $\mu_{c'}\circ \sigma = \sigma \circ \mu_c$. By induction, we can then show that $\mu^n_{c'}(0) = \sigma(\mu^n_c(0))$ for all $n\ge 0$. Hence, $\infw{u}_{c'} = \sigma(\infw{u}_c)$.
\end{proof}

Combined with Corollary~\ref{C:automatic sequences}, this result implies that, if $c_0 \cdots c_{k-2}(c_{k-1} - 1)$ is maximal among its conjugates, the word $\infw{u}$ is simple-Parry automatic in the sense that it is automatic for the integer numeration system classically associated with a simple Parry number.

\begin{example}
We illustrate Proposition~\ref{pro: link with simple Parry} by resuming Example~\ref{Ex: two NS for Fibonacci}.
We have $c_0c_1c_2(c_3-1)=1010=v^2$ with $v=10$ and $c'=11$.
The corresponding simple Parry number is the Golden ratio $\varphi$.
Moreover, if $\sigma\colon i \mapsto i \bmod 2$, then 
\[
    \sigma(\mu_{c}^\omega(0)) = \sigma(0120301001201 \cdots) = 0100101001001 \cdots
\]
is the Fibonacci word.
\end{example}

\section{String attractors of the prefixes}\label{sec: SA}

We now turn to the concept of string attractors in relation to the fixed points of the morphisms $\mu_c$, $c \geq 0^k$.
A \emph{string attractor} of a finite word $y = y_1 \cdots y_n$ is a set $\Gamma \subseteq \{1, \ldots, n\}$ such that every non-empty factor of $y$ has an occurrence crossing a position in $\Gamma$, i.e., for each factor $x \in A^m$ of $y$, there exists $i \in \Gamma$ and $j$ such that $i \in \{j, \ldots, j + m - 1\}$ and $x = y[j,j + m)$.

\begin{example}
The set $\{2,3,4\}$ is a string attractor of the word $0\,\underline{1}\,\underline{2}\,\underline{0}\,0\,1$. Indeed, it suffices to check that the factors $0$, $1$ and $01$ have an occurrence crossing one of the underlined positions. No smaller string attractor exists since at least one position in the set is needed per different letter in the word.
\end{example}

\noindent\textbf{Warning.}
We would like to stress the following crucial point: in this paper, the letters of infinite words are indexed starting from $0$ while the positions in a string attractor are counted starting at $1$.
This could be seen as confusing, but we use the same notation as the original paper on string attractors \cite{Kampa-Prezza}.
Where ambiguity may occur, we explicitly declare how finite words are indexed.

The family of words $\infw{u}_c$ contains the famous $k$-bonacci words (when $c = 1^k$), and it is known for these words that the positions in $\{U_n : n \geq 0\}$ are sufficient to find string attractors of minimal sizes~\cite{CGRRSS23}. It is thus natural to wonder if it is also the case for each word $\infw{u}_c$. We first obtain the following result.

\begin{proposition}\label{P:string attractor imply fractional powers}
Let $c$ satisfy~\eqref{eq:WH}. If every prefix of $\infw{u}$ has a string attractor made of elements of $\{\length{n} : n \geq 0\}$, then $\infw{u}[0,\length{n+1}-1)$ is a fractional power of $\word{n}$ for all $n \geq 0$.
\end{proposition}
\begin{proof}
Assume to the contrary that there exists $N$ such that $\infw{u}[0,\length{N+1}-1)$ is not a fractional power of $\word{N}$. Therefore, let us denote $\word{N}x_Nb_N$, where $x_N$ is a possibly empty word and $b_N$ is a letter, the shortest prefix of $\infw{u}[0,\length{N+1}-1)$ that is not a fractional power of $\word{N}$. We show that $x_Nb_N$ is not a factor of $\word{N}x_N$, and therefore that any string attractor of $\word{N}x_Nb_N$ must contain a position in $[\length{N} + 1, \length{N+1} - 1]$, which contradicts the assumption.

More generally, we show that, for all $n$, if $\word{n}x_nb_n$ is the shortest prefix of $\infw{u}$ that is not a fractional power of $\word{n}$, then either $x_n = \eps$ or it is always followed in $\word{n}x_n$ by the letter $b \ne b_n$ such that $\word{n}x_nb$ is a prefix of $\word{n}^\omega$. Observe that, in the first case, by Proposition~\ref{P:alternative definition of the words}, $c_0 = 1$ and $c_i = 0$ for all $1 \leq i \leq n$. In other words, $\word{n} = 012 \cdots n$ and $b_n = n+1$. Therefore we can indeed conclude that $x_nb_n$ is not a factor of $u_nx_n$ in both cases.

We proceed by induction on $n$. If $n = 0$, then $\word{n}=0$ and $\word{n}x_nb_n = 0^{c_0}1$, so $x_n = 0^{c_0 - 1}$ is always followed by $b=0$ in $\word{n}x_n$.

Assume now that the claim is true for $n-1$ and let us prove it for $n$. If $x_{n-1} = \eps$, then, as above, $\word{n} = 012 \cdots n$. Therefore, $\word{n+1} = \mu(\word{n}) = 012 \cdots n \mu(n)$ and $x_n$ is either $\eps$ (if $\mu(n) = (n+1)$) or $0$. In the second case, we have $b_n \in \{0, n+1\}$ and $x_n$ is only followed by $1 \ne b_n$ in $\word{n}x_n$ so we conclude.

If $x_{n-1} \ne \eps$, then $\mu(\word{n-1}x_{n-1})$ is a prefix of $\infw{u}$ that is a fractional power of $\word{n}$, so $\mu(x_{n-1})$ is a prefix of $x_n$. Notice that $x_{n-1}$ cannot end with $(k-1)$. Indeed, this follows from Remark~\ref{R:extensions of letters} since $\word{n-1}x_{n-1}$ is the longest common prefix between $\infw{u}$ and $\word{n-1}^\omega$. This implies that any non-suffix occurrence of $x_n$ in $\word{n}x_n$ comes from a non-suffix occurrence of $x_{n-1}$ in $\word{n-1}x_{n-1}$. By the induction hypothesis, such an occurrence of $x_{n-1}$ is always followed by $b \ne b_{n-1}$ such that $\word{n-1}x_{n-1}b$ is a prefix of $\word{n-1}^\omega$. Therefore, if $0^\ell$ is the longest common prefix between $\mu(b)$ (resp., $\mu(b0)$ if $b = k-1$) and $\mu(b_{n-1})$ (resp., $\mu(b_{n-1}0)$ if $b_{n-1} = k-1$), then $x_n = \mu(x_{n-1}) 0^\ell$ and $x_n$ is always followed in $\word{n}x_n$ by $c \ne b_n$ such that $\mu(\word{n-1}x_{n-1})0^\ell c = \word{n}x_nc$ is a prefix of $\word{n}^\omega$.
\end{proof}

Based on this result, we see that fractional powers of the word $\word{n}$ will play a key role in determining string attractors.
We thus introduce some notations.

\begin{definition}
\label{D: longest fractional power prefix}
Let $c$ satisfy~\eqref{eq:WH}. For all $n \geq 0$, we let $\upperw{n}$ denote the longest prefix of $\infw{u}$ that is a fractional power of $\word{n}$, i.e., the longest common prefix between $\infw{u}$ and $(\word{n})^\omega$.
For all $n \geq 0$, we also let $\upperb{n} = |\upperw{n}|$.
\end{definition}

Working with fractional powers also has another advantage from the string attractor point of view. Indeed, there is no trivial link in general between the string attractors of the finite words $w$ and $wa$, where $a$ is a letter. However, we have the following result which can be derived from the proofs of \cite[Propositions 12 and 15]{saandcow}.

\begin{proposition}
\label{P:sa of fractional powers}
Let $z$ be a non-empty word and let $x$ and $y$ be fractional powers of $z$ with $|z| \leq |x| \leq |y|$. If $\Gamma$ is a string attractor of $x$, then $\Gamma \cup \{|z|\}$ is a string attractor of $y$.
\end{proposition}

Motivated by this result, to describe string attractors of each prefix, it is now sufficient to be able to describe, for all $n \geq 1$, a string attractor of a prefix of length $m_n$ for some $m_n \in [\length{n} - 1, \upperb{n-1}]$. This argument is the key in the proof of our main theorem.

For $n \geq 0$, we denote
\[
    \Gamma_n =
    \begin{cases}
    \{\length{0}, \ldots, \length{n}\}, & \text{if } 0\le n \leq k-1;\\
    \{\length{n-k+1}, \ldots, \length{n}\}, & \text{if } n \geq k.
    \end{cases}
\]
We also define
\begin{equation}
\label{eq:def-of-Pn}
        \lowerb{n} =
    \begin{cases}
    \length{n}, & \text{if } 0\le n \leq k-1;\\
    \length{n} + \length{n-k+1} - \length{n-k} - 1, & \text{if } n \geq k.
    \end{cases}
\end{equation}

The next lemma directly follows from the definition of $\lowerb{n}$ and from Proposition~\ref{P:alternative definition of the words}.

\begin{lemma}
\label{L:inequality with lower and upper}
Let $c$ satisfy~\eqref{eq:WH}. Then $\lowerb{n} \leq \length{n+1} - 1$ for all $n \geq 0$.
\end{lemma}

To simplify the statement of the following theorem, we set $\Gamma_{-1} = \emptyset$.
\begin{theorem}
\label{T:sa of prefixes}
Let $c = c_0 \cdots c_{k-1} \geq 0^k$ with $c_0, c_{k-1} \geq 1$ and such that $\infw{u}[0,U_{n+1} - 1)$ is a fractional power of $u_n$, and let $n \geq 0$.
\begin{enumerate}
    \item If $m \in [\length{n}, \upperb{n}]$, then $\Gamma_{n-1} \cup \{\length{n}\}$ is a string attractor of $\infw{u}[0,m)$.
    \item If $m \in [\lowerb{n},\upperb{n}]$, then $\Gamma_n$ is a string attractor of $\infw{u}[0,m)$.
\end{enumerate}
\end{theorem}
\begin{proof}
Let us simultaneously prove the two claims by induction on $n$.
If $n = 0$, then $1 \leq m \leq c_0$, so $\infw{u}[0,m) = 0^m$ and the conclusion directly follows for both claims.
Assume now that the claims are satisfied for $n-1$ and let us prove them for $n$. By Lemma~\ref{L:inequality with lower and upper}, $\lowerb{n-1} \leq \length{n} -1$, and by hypothesis on $c$, $\length{n} - 1 \leq \upperb{n-1}$. Therefore, by the induction hypothesis, $\Gamma_{n-1}$ is a string attractor of $\infw{u}[0,\length{n} - 1)$. This implies that $\Gamma_{n-1} \cup \{\length{n}\}$ is a string attractor of $\word{n}$ so, by Proposition~\ref{P:sa of fractional powers}, of $\infw{u}[0,m)$ for all $m \in [\length{n}, \upperb{n}]$. This ends the proof of the first claim.

Let us now prove the second claim. Observe that, using Proposition~\ref{P:sa of fractional powers}, it suffices to prove that $\Gamma_n$ is a string attractor of $\infw{u}[0,P_n)$.
If $0\le n \leq {k-1}$, then $\Gamma_n = \Gamma_{n-1} \cup \{\length{n}\}$ so we can directly conclude using the first claim.
Thus assume that $n \geq k$. Then by the first claim, $\Gamma_n \cup \{\length{n-k}\} = \Gamma_{n-1} \cup \{\length{n}\}$ is a string attractor of $\infw{u}[0,P_n)$. Therefore, it remains to show that the position $\length{n-k}$ is not needed in the string attractor. In other words, we prove that the factors of $\infw{u}[0,P_n)$ having an occurrence crossing position $\length{n-k}$ (and no other position of $\Gamma_n \cup \{\length{n-k}\}$) have another occurrence crossing a position in $\Gamma_n$. More precisely, we show that they have an occurrence crossing position $\length{n}$.
To help the reader with the proof, we illustrate the situation in Figure~\ref{fig:proof of sa}.

As the smallest position in $\Gamma_n$ is $\length{n-k+1}$, we need to consider the factor occurrences crossing position $\length{n-k}$ in $\infw{u}[0,\length{n-k+1} - 1)$. So, if we write $\infw{u}[0,P_n) = \word{n} w$, it is sufficient to show that $\word{n-k}$ is a suffix of $\word{n}$ and that $w' := \infw{u}[\length{n-k},\length{n-k+1} - 1)$ is a prefix of $w$.
Observe that 
\begin{equation}
\label{eq:length-suffix-w}
    |w| = \lowerb{n} - \length{n} = \length{n-k+1} - \length{n-k} - 1
\end{equation}
by definition of $\lowerb{n}$, so $|w'| = |w|$ and we actually show that $w' = w$.

\begin{figure}[h!t]
    \centering
    \begin{tikzpicture}[cross/.style={path picture={ 
  \draw[black]
(path picture bounding box.south east) -- (path picture bounding box.north west) (path picture bounding box.south west) -- (path picture bounding box.north east);
}}]  
        \node (\lowerb{n}) at (-1,0.4) {$\infw{u}[0,\lowerb{n})=$};
        \node [fit={(0,0) (7,0.8)}, inner sep=0pt, draw=black, thick] (u_n) {};
        \node [fit={(7,0) (8.5,0.8)}, inner sep=0pt, draw=black, thick, fill=gray!30] (u_n-k-1) {};

        \node [thick] (label_u_n) at (3.5,1) {$\word{n}$};
        \node [thick] (label_w) at (7.75,1) {$w$};
        
        \draw [thick, decoration={brace, mirror, raise=0.05cm, amplitude=5pt}, decorate] (0,0) -- node[below=.2cm] {$\word{n-k}$} (2,0) {};
        \node [circle, fill=red,inner sep=0pt,minimum size=4pt] (T_n-k) at (1.8,0.5) {};
        \node [cross] (T_n-k bis) at (1.8,0.5) {};
        \node [red] (label_T_n-k) at (1.8,0.2) {$\length{n-k}$};
        
        \draw [thick, decoration={brace, mirror, raise=0.05cm, amplitude=5pt}, decorate] (2,0) -- node[below=.2cm] {$w'$} (3.5,0) {};        
        \node [circle,fill=red,inner sep=0pt,minimum size=4pt] (T_n-k+1) at (3.7,0.5) {};
        \node [red] (label_T_n-k) at (3.7,0.2) {$\length{n-k+1}$};
        \draw [thick, decoration={brace, raise=0.05cm, amplitude=5pt}, decorate] (0,0.8) -- node[above=.2cm] {$w'$} (1.5,0.8) {};
        
        \draw [thick, decoration={brace, mirror, raise=0.05cm, amplitude=5pt}, decorate] (5,0) -- node[below=.2cm] {$\word{n-k}$} (7,0) {};
        \draw [thick, decoration={brace, mirror, raise=0.05cm, amplitude=5pt}, decorate] (7,0) -- node[below=.2cm] {$w'$} (8.5,0) {};
        \node [circle,fill=red,inner sep=0pt,minimum size=4pt] (T_n) at (6.7,0.5) {};
        \node [red] (label_T_n) at (6.7,0.2) {$\length{n}$};
    \end{tikzpicture}
    \caption{Representation of the proof of the second claim of Theorem~\ref{T:sa of prefixes}. As we warned the reader before, elements in a string attractor are indexed starting at $1$ (in red), while indices of letters in $\infw{u}$ start at $0$.}
    \label{fig:proof of sa}
\end{figure}

The fact that $\word{n-k}$ is a suffix of $\word{n}$ is a direct consequence of Proposition~\ref{P:alternative definition of the words} as $c_{k-1} \geq 1$ by assumption.
Since $\infw{u}[0,U_{n-k+1}-1)$ is a fractional power of $\word{n-k}$ by assumption, $w'$ is a prefix of $\word{n}$. By Lemma~\ref{L:inequality with lower and upper} and by assumption, we also have $\lowerb{n} \leq \length{n+1}-1 \leq \upperb{n}$, so $\infw{u}[0,\lowerb{n})$ is a fractional power of $\word{n}$. This implies that $w = w'$.
\end{proof}

\section{Fractional power prefixes and anti-Lyndon words}
\label{sec: frac power and anti-Lyndon}

In this section, we study the words $\upperw{n}$ and their lengths $\upperb{n}$. As we will show in Proposition~\ref{P:longest fractional power}, these words have a particular structure related to (anti-)Lyndon words.
To prove this, we introduce some more notations.
For all $n \geq 0$, the pair $\{i_n, j_n\}$ designates the two (distinct) letters following $\upperw{n}$ in $\infw{u}$ and in $(\word{n})^\omega$. 
Without loss of generality, we always assume that $i_n < j_n$.

\begin{example}\label{ex:c102 - i_nj_n}
Set $c=102$.
Recall from Example~\ref{ex:c102 - words and lengths} that the first few words in $(\word{n})_{n\geq0}$ are 
$0$, $01$, $012$, $01200$, $012000101$, $012000101012012$. It is then easy to see that the first few words in $(\upperw{n})_{n\geq0}$ are $0$, $01$, $012 0$, $01200 01$, $012000101 0120$.
So we conclude that the first few pairs in $(\{i_n, j_n\})_{n\geq0}$ are $\{0,1\}$, $\{0,2\}$, $\{0,1\}$, $\{0,2\}$, $\{0,1\}$.
\end{example}

The following lemma gives a recursive construction for the sequences $(i_n)_{n \geq 0}$ and $(j_n)_{n \geq 0}$, as well as a first structure for the words $\upperw{n}$.

\begin{lemma}
\label{L:construction of i, j, l}
Let $c$ satisfy~\eqref{eq:WH}.
For all $n \geq 0$, we have $\upperw{n} = \word{n}^{\ell_0} \word{n-1}^{\ell_1} \cdots \word{0}^{\ell_n}$ where the sequences $(\ell_n)_{n \geq 0}$, $(i_n)_{n \geq 0}$, $(j_n)_{n \geq 0}$ are recursively constructed as follows:
$\ell_0 = c_0$, $i_0 = 0$, $j_0 = 1$, and for all $n \geq 0$, if $j_n \leq k-2$, we have
\[
    \{\ell_{n+1}, i_{n+1}, j_{n+1}\} =
    \begin{cases}
    \{c_{j_n}, 0, j_n + 1\}, & \text{if } c_{i_n} > c_{j_n};\\
    \{c_{j_n}, i_n + 1, j_n + 1\}, & \text{if } c_{i_n} = c_{j_n};\\
    \{c_{i_n}, 0, i_n + 1\}, & \text{if } c_{i_n} < c_{j_n};
    \end{cases}
\]
and if $j_n = k-1$, we have $\{\ell_{n+1}, i_{n+1}, j_{n+1}\} = \{c_{i_n}, 0, i_n + 1\}$.
\end{lemma}
\begin{proof}
We prove the claimed structure for the sequences $(\ell_n)_{n \geq 0}$, $(i_n)_{n \geq 0}$, $(j_n)_{n \geq 0}$ and also that $c_0 = \max\{c_0, \ldots, c_{j_n - 1}\}$ for all $n \geq 0$ by induction.

For the base case $n = 0$, as $\word{0} = 0$ and $\word{1} = 0^{c_0} 1$ is a prefix of $\infw{u}$, we directly have $\ell_0 = c_0$, $i_0 = 0$, $j_0 = 1$ and $c_0 = \max\{c_0\}$.

Let us now move to the induction step: assume that both claims are satisfied for $n$ and let us prove them for $n + 1$.
For the first claim, by definition, $\mu(\upperw{n})$ is a prefix of both $\mu(\infw{u})=\infw{u}$ and $\mu(\word{n})^\omega = (\word{n+1})^\omega$. Moreover, it is followed in one of them by $\mu(i_n) = 0^{c_{i_n}} \cdot (i_n + 1)$ and in the other by $\mu(j_n)$.
The image of $j_n$ under $\mu$ takes two forms.

If $j_n \leq k - 2$, then $\mu(j_n) = 0^{c_{j_n}} \cdot (j_n + 1)$. Thus, as $i_n + 1 \ne j_n + 1$, we have $\upperw{n+1} = \mu(\upperw{n}) 0^{\ell_{n+1}}$ where $0^{\ell_{n+1}}$ is the longest common prefix between $\mu(i_n)$ and $\mu(j_n)$. We then have
\[
    \{\ell_{n+1}, i_{n+1}, j_{n+1}\} =
    \begin{cases}
    \{c_{j_n}, 0, j_n + 1\}, & \text{if } c_{i_n} > c_{j_n};\\
    \{c_{j_n}, i_n + 1, j_n + 1\}, & \text{if } c_{i_n} = c_{j_n};\\
    \{c_{i_n}, 0, i_n + 1\}, & \text{if } c_{i_n} < c_{j_n}.
    \end{cases}
\]
The conclusion of the first claim follows from the fact that $\mu(\upperw{n}) = \word{n+1}^{\ell_0} \cdots \word{1}^{\ell_n}$ by the induction hypothesis.

If $j_n = k - 1$, then by Remark~\ref{R:extensions of letters}, $\upperw{n+1}$ is not only followed by $\mu(k-1)$ but by $\mu(k-1) \mu(0) = 0^{c_{k-1} + c_0} \cdot 1$. By the second claim, we have 
\[
c_{i_n} \leq \max\{c_0, \ldots, c_{k-2}\} = c_0 < c_{k-1} + c_0
\]
as $c_{k-1} \geq 1$ by assumption. We conclude that $\{\ell_{n+1}, i_{n+1}, j_{n+1}\} = \{c_{i_n}, 0, i_n + 1\}$.

The second claim is also satisfied as $\max\{c_0, \ldots, c_{j_{n+1} - 1}\} \leq \max\{c_0, \ldots, c_{j_n - 1}\}$. Indeed, in all cases, either $j_{n+1} \leq j_n$, or $j_{n+1} = j_n + 1$ and $c_{j_n} \leq \max\{c_0, \ldots, c_{j_n - 1}\}$.
\end{proof}

\begin{example}
\label{Ex:construction of i, j, l}
Let us take $c = 210221$ for which $k = 6$. The first few elements of the sequences $(\ell_n)_{n \geq 0}$, $(i_n)_{n \geq 0}$, $(j_n)_{n \geq 0}$ are given in Table~\ref{tab:Duval-inverse-order}.
We already observe that they are (eventually) periodic. Indeed, $\{i_1, j_1\} = \{0, 2\} = \{i_4, j_4\}$ and, as $\{i_N, j_N\}$ entirely determines the rest of the sequences, $(\ell_n)_{n \geq 0}$, $(i_n)_{n \geq 0}$, $(j_n)_{n \geq 0}$ are eventually periodic of period length 3 starting from index 1 (and even from index 0 for $(\ell_n)_{n \geq 0}$).
\begin{table}
    \caption{Illustration of the construction of the sequences $(\ell_n)_{n \geq 0}$, $(i_n)_{n \geq 0}$, $(j_n)_{n \geq 0}$ in the case where $c = 210221$.}
    \label{tab:Duval-inverse-order}
    \begin{tabular}{c|c|c|c}
        $n$ & Comparison & $\ell_n$ & $\{i_n,j_n\}$ \\[0.5pt]
        \hline
        \rule{0pt}{.8\normalbaselineskip}
        \!\!$0$ & / & $c_0 = 2$ & $\{0,1\}$ \\
        $1$ & $c_0 > c_1$ & $c_1 = 1$ & $\{0,2\}$ \\
        $2$ & $c_0 > c_2$ & $c_2 = 0$ & $\{0,3\}$ \\
        $3$ & $c_0 = c_3$ & $c_3 = 2$ & $\{1,4\}$ \\
        $4$ & $c_1 < c_4$ & $c_1 = 1$ & $\{0,2\}$ \\
        $5$ & $c_0 > c_2$ & $c_2 = 0$ & $\{0,3\}$ \\
        $6$ & $c_0 = c_3$ & $c_3 = 2$ & $\{1,4\}$   
    \end{tabular}
\end{table}
\end{example}

From the recursive definition given in Lemma~\ref{L:construction of i, j, l}, we derive the following result.

\begin{lemma}
\label{L:i, j and borders}
Let $c$ satisfy~\eqref{eq:WH}.
For all $n \geq 0$, the word $c_0 \cdots c_{i_n - 1}$ is a border of the word $c_0 \cdots c_{j_n - 1}$, i.e., $c_0 \cdots c_{i_n - 1} = c_{j_n - i_n} \cdots c_{j_n - 1}$.
\end{lemma}
\begin{proof}
Once again, we prove the result by induction on $n \geq 0$. Notice that, if $i_n = 0$, then the word $c_{j_n - i_n} \cdots c_{j_n - 1}$ is empty, hence the conclusion. This is in particular the case for $n = 0$. Assume now that the claim holds for $n$ and let us prove it for $n+1$. By Lemma~\ref{L:construction of i, j, l}, we have $i_{n+1} = 0$ unless $c_{i_n} = c_{j_n}$. In this case, $i_{n+1} = i_n + 1$ and $j_{n+1} =  j_n + 1$ so, as $c_0 \cdots c_{i_n - 1} = c_{j_n - i_n} \cdots c_{j_n - 1}$ by the induction hypothesis, we directly have $c_0 \cdots c_{i_{n+1} - 1} = c_{j_{n+1} - i_{n+1}} \cdots c_{j_{n+1} - 1}$.
\end{proof}

We now show the link with (anti-)Lyndon words. Before doing so, we recall some famous properties of Lyndon words that will be useful.
The first result is part of the folklore, but a proof can be found, for instance, in~\cite{DUVAL20082261}.

\begin{proposition}\label{P:lyndon are unbordered}
Lyndon words are unbordered, i.e., if $w$ is a both a prefix and a suffix of a Lyndon word $v$, then $w = \eps$ or $w = v$.
\end{proposition}

The next result is shown within the proof of the Chen-Fox-Lyndon Theorem (Theorem~\ref{thm:CFL-fact}). See, for instance, \cite[Theorem 5.1.5]{Lothaire97}.

\begin{proposition}
\label{P:longest Lyndon prefix}
Let $w\in A^*$ be a non-empty word and let $(\ell_1,\cdots, \ell_n)$ be its Lyndon factorization as in Theorem~\ref{thm:CFL-fact}.
Then $\ell_1$ is the longest Lyndon prefix of $w$.
\end{proposition}

Duval provided an algorithm computing the Lyndon factorization of a word in linear time~\cite{Duval83}. It is based on a decomposition of the word into three parts $xyz$: we already computed the Lyndon factorization of $x$ and we are now looking at $w=yz$, where $y$ is a fractional power of a Lyndon word $v$ and $z$ is the part that we still need to explore. We keep track of the position of the first letter of $z$ with an index $j$, and of the period of $y$ (i.e., the length of $v$) using an index $i$ such that $j - i = |v|$.

\begin{algo}[Duval~\cite{Duval83}]
\label{A: Duval}
Let $(A, \leq)$ be an ordered set and let $w = w_0 \cdots w_n$ be a length-$n$ word over $A$. We denote $w_{n+1}$ a new symbol smaller than all the letters of $w$.
Set $i = 0$ and $j = 1$.
While $i \leq n$, compare $w_i$ and $w_j$ and do the following:
\begin{itemize}
\item
    if $w_i < w_j$, then set $j = j+1$ and $i = 0$;
\item
    if $w_i = w_j$, then set $j = j+1$ and $i = i+1$;
\item
    if $w_i > w_j$, then output $w_0 \cdots w_{j-i-1}$ as the next element in the Lyndon factorization and restart the algorithm with the word $w_{j-i} \cdots w_n$.
\end{itemize}
\end{algo}

Using the notation of the paragraph preceding Algorithm~\ref{A: Duval}, we explain the three cases present in the algorithm.
We want to compute the next Lyndon word in the Lyndon factorization of a word, knowing that of some of its prefixes.
By definition of $i$ and $j$, we compare the letter $w_j$ in $z$ with the letter $w_i$, spaced by $|v|$ letters.
\begin{itemize}
\item
    If $w_i < w_j$,
    then $y w_j$ is a Lyndon word by~\cite[Lemme 2]{Duval80}, so we update $y$ to $y w_j$ and $v$ to $y$. 
\item    
    If $w_i = w_j$,
    then $y w_j$ is still a fractional power of $v$, so we simply update $y$ to $y w_j$ without changing the length of $v$ (that is, we do not modify $j-i$).
\item
    If $w_i > w_j$, then $y w_j$ cannot be a prefix of a Lyndon word, so the longest Lyndon prefix of $w$ is $v$.
\end{itemize}

We are now ready to prove the structure of the words $\upperw{n}$ and its link with anti-Lyndon words.

\begin{proposition}
\label{P:longest fractional power}
Let $c$ satisfy~\eqref{eq:WH}.
Define $\infw{a}$ as the infinite concatenation of the longest anti-Lyndon prefix of the word $c_0 \cdots c_{k-2}$. Then for all $n\geq0$, $\upperw{n} = \word{n}^{\infw{a}_0} \word{n-1}^{\infw{a}_1} \cdots \word{0}^{\infw{a}_n}$.
In particular, $\upperb{n} = \sum_{i=0}^n \infw{a}_i \length{n-i}$.
\end{proposition}
\begin{proof}
By Lemma~\ref{L:construction of i, j, l}, the beginning of the construction of the sequences $(\ell_n)_{n \geq 0}$, $(i_n)_{n \geq 0}$, $(j_n)_{n \geq 0}$ corresponds exactly to the first application of Duval's algorithm to the word $c_0 \cdots c_{k-2}$ with the order $\invlex{\leq}$. 
More specifically, letting $N$ denote the first index $n$ for which $c_{i_n} < c_{j_n}$ or $j_n = k-1$ and setting $p = j_N - i_N$, then Duval's algorithm for $\invlex{\leq}$ implies that the word $\ell_0 \cdots \ell_{p-1}$ is the first element in the Lyndon factorization of $c_0 \cdots c_{k-2}$ for the order $\invlex{\leq}$.
Therefore, $\ell_0 \cdots \ell_{p-1} = c_0 \cdots c_{p-1}$ is the longest anti-Lyndon prefix of $c_0 \cdots c_{k-2}$ by Proposition~\ref{P:longest Lyndon prefix}. Let us denote it $v$.
As in the statement, let $\infw{a}=vvv\cdots$.

Observe that, by definition of $N$ and by Lemma~\ref{L:construction of i, j, l}, for all $1 \leq n \leq N$, we have $j_n = n + 1$ as it is incremented at each step, and $\ell_{n} = c_{j_{n-1}} = c_n$.
In particular, $p = j_N - i_N = N + 1 - i_N$.

We now prove that $\ell_n = \infw{a}_n$ for all $n \geq 0$.
By definition of $\infw{a}$, the equality holds for $0\le n<p$, so it is enough to look at all $n \geq p$.
We show by induction on $n \geq p$ that $\ell_n = c_{n \bmod p}$, $j_n \equiv (n + 1) \bmod p$, and $j_n \leq N + 1$.

For $p \leq n \leq N$, we already have $\ell_n = c_n$, $j_n = n + 1$, and $j_n \leq N + 1$ by the observation made above. 
Moreover, Duval's algorithm implies that $c_0 \cdots c_N$ is periodic of period length $p$, so  $\ell_n = c_n = c_{n \bmod p}$.
This is also true for $n = N + 1$ as $N + 1 = p + i_N \equiv i_N \bmod p$. Indeed, by Lemma~\ref{L:construction of i, j, l} and by definition of $N$, we have $\ell_{N + 1} = c_{i_N} = c_{N + 1 \bmod p}$ and 
\begin{equation}
\label{Eq: case N}
    j_{N + 1} = i_N + 1 \equiv N + 2 \bmod p.
\end{equation}

Assume now that the claim is true for indices up to $n \geq N + 1$ and let us prove it for $n+1$. By the induction hypothesis, we have $j_n \leq N + 1$, so we distinguish two cases.

\textbf{Case 1.}
    If $j_n \leq N$, then $j_n\leq k - 2$ (as $j_N = N+1 \le k-1$). By Lemma~\ref{L:i, j and borders}, comparing $c_{i_n}$ and $c_{j_n}$ is equivalent to comparing $c_0 \cdots c_{i_n}$ and $c_{j_n - i_n} \cdots c_{j_n}$. As mentioned earlier in the proof, $c_0 \cdots c_N$ is a fractional power of $v$, so $c_0 \cdots c_{i_n}$ is a prefix of a power of $v$ while $c_{j_n - i_n} \cdots c_{j_n}$ is a prefix of a power of a conjugate of $v$. As $v$ is Lyndon for $\invlex{\leq}$, its powers are smaller than the powers of its conjugates for $\invlex{\leq}$, thus $c_0 \cdots c_{i_n} \invlex{\leq} c_{j_n - i_n} \cdots c_{j_n}$ and $c_{i_n} \leq_- c_{j_n}$, i.e., $c_{i_n} \geq c_{j_n}$. Using Lemma~\ref{L:construction of i, j, l}, we conclude that $\ell_{n+1} = c_{j_n} = c_{n+1 \bmod p}$ as $j_n \leq N$ is congruent to $n+1 \bmod p$ by the induction hypothesis and $c_0 \cdots c_N$ has period length $p$. We also have $j_{n+1} = j_n + 1$ thus $j_{n+1} \leq N + 1$ and $j_{n+1} \equiv n + 2 \bmod p$.

\textbf{Case 2.}
    If $j_n = N + 1$, then using Lemma~\ref{L:i, j and borders}, we know that $c_0 \cdots c_N = c_0 \cdots c_{j_n - 1}$ has a border of length $i_n$ so $c_0 \cdots c_N$ has period length $N + 1 - i_n$. Since it also has period length $p$ and $c_0 \cdots c_{p-1}$ is anti-Lyndon thus unbordered by Proposition~\ref{P:lyndon are unbordered}, we must have that $N + 1 - i_n$ is a multiple of $p = N + 1 - i_N$. In other words,
    \begin{equation}
    \label{Eq: congruence iN and in}
        i_n \equiv i_N \bmod p.
    \end{equation}
    In particular, by periodicity, $c_{i_n} = c_{i_N}$.
    Moreover, $j_n = N+1 = j_N$ so $\{c_{i_n}, c_{j_n}\} = \{c_{i_N}, c_{j_N}\}$. Therefore, by Lemma~\ref{L:construction of i, j, l} and by definition of $N$, we have 
    \begin{equation}
    \label{Eq: comparison l and j between n et N}
        \ell_{n+1} = \ell_{N+1} \quad \text{and} \quad
    j_{n+1} = i_n + 1 \leq N + 1.
    \end{equation}
    By the induction hypothesis for $n$, we have 
    \begin{equation}
    \label{Eq: congruence N+1 and n+1}
        N + 1 = j_n \equiv n + 1 \bmod p. 
    \end{equation}   
    We conclude that 
    \[
        \ell_{n + 1} = \ell_{N + 1} = c_{(N + 1) \bmod p} = c_{(n + 1) \bmod p},
    \]
    where the first equality follows by~\eqref{Eq: comparison l and j between n et N}, the second by the induction hypothesis for $N+1$, and the last by Congruence~\eqref{Eq: congruence N+1 and n+1}, and
    \[
        j_{n+1} = i_n + 1 \equiv i_N + 1 \equiv N + 2 \equiv n + 2 \bmod p,
    \]
    where the first equality follows from~\eqref{Eq: comparison l and j between n et N}, the first congruence from~\eqref{Eq: congruence iN and in}, the second by~\eqref{Eq: case N}, and the last by Congruence~\eqref{Eq: congruence N+1 and n+1}.
    This ends the proof.
\end{proof}

\begin{example}
\label{ex:c102 - frac pow prefixes}
Let us pursue Example~\ref{ex:c102 - i_nj_n} for which $c=102$.
The first few words in $(\upperw{n})_{n\geq0}$ are $0, 01, 012 0, 01200 01, 012000101 0120$.
The longest anti-Lyndon prefix of $c_0c_1=10$ is $10$ itself so $\infw{a} = (10)^\omega$. We can easily check that the first few $q_n$'s indeed satisfy Proposition~\ref{P:longest fractional power}.
\end{example}

Now that we have a good understanding of the fractional powers of the words $\word{n}$, we look for an equivalent description of the condition in Proposition~\ref{P:string attractor imply fractional powers} and Theorem~\ref{T:sa of prefixes}.
This is the purpose of Proposition~\ref{pro:equivalent-conditions-4}, but we first need the following technical lemma.

\begin{lemma}
\label{L:inequality of c_0...c_k-2}
    Let $c$ satisfy~\eqref{eq:WH}  and let $w$ denote the longest anti-Lyndon prefix of $c_0 \cdots c_{k-2}$.
    \begin{enumerate}
        \item Then $c_0 \cdots c_{k-2} \geq \infw{a}[0, k-2]$.
        \item Moreover, $c_0 \cdots c_{k-2}(c_{k-1}-1)$ is maximal among its conjugates if and only if the following three assertions hold:
        \begin{enumerate}
            \item\label{lem:iff-assumption-i} We have $c_0 \cdots c_{k-2} = \infw{a}[0, k-2]$.
            \item\label{lem:iff-assumption-ii} We have $c_{k-1} - 1 \leq \infw{a}_{k-1}$.
            \item\label{lem:iff-assumption-iii} If $c_{k-1} - 1 = \infw{a}_{k-1}$, then $c_0 \cdots c_{k-2}(c_{k-1}-1)$ is an integer power of $w$.
        \end{enumerate}
    \end{enumerate}
\end{lemma}

\begin{proof}
We show the first claim. Assume by contradiction that there exists a minimal index $i\in\{|w|,\ldots,k-2\}$ such that $c_0 \cdots c_i < \infw{a}[0,i]$. Then $c_0 \cdots c_i = w^\ell va$ with a proper prefix $v$ of $w$ and a letter $a$ such that $va < w$. So~\cite[Lemme 2]{Duval80} implies that $c_0 \cdots c_i$ is an anti-Lyndon prefix of $c_0 \cdots c_{k-2}$. As $i \geq |w|$, this contradicts the maximality of $w$.

We now turn to the second claim. Assume that $c_0 \cdots c_{k-2}(c_{k-1}-1)$ is maximal among its conjugates.
We first show that $c_0 \cdots c_{k-2} (c_{k-1} - 1) \leq \infw{a}[0,k-1]$. If it is not the case, there exist $\ell \geq 1$, a proper prefix $u$ of $w$, a letter $a$ and a word $v$ such that $c_0 \cdots c_{k-2} (c_{k-1} - 1) = w^\ell u a v$ and $ua > w$. Then $u a v w^\ell > c_0 \cdots c_{k-2} (c_{k-1} - 1)$, so $c_0 \cdots c_{k-2} (c_{k-1} - 1)$ is not maximal among its conjugates. This is a contradiction. 
Therefore, we have $c_0 \cdots c_{k-2} (c_{k-1} - 1) \leq \infw{a}[0,k-1]$.
Using the first claim, we get $c_0 \cdots c_{k-2} = \infw{a}[0,k-2]$ and $c_{k-1} - 1 \leq \infw{a}_{k-1}$, which gives Items~\ref{lem:iff-assumption-i} and~\ref{lem:iff-assumption-ii}.
Now if $c_{k-1}-1 = \infw{a}_{k-1}$, then $c_0 \cdots c_{k-2}(c_{k-1}-1)$ is a fractional power of $w$.
It is moreover maximal among its conjugates by assumption so, since $w$ is anti-Lyndon, $c_0 \cdots c_{k-2}(c_{k-1}-1)$ is in fact an integer power of $w$.
This proves Item~\ref{lem:iff-assumption-iii}. 

Let us now suppose that Assumptions~\ref{lem:iff-assumption-i}, ~\ref{lem:iff-assumption-ii}, and ~\ref{lem:iff-assumption-iii} hold.
Let us consider a conjugate $c_i \cdots c_{k-2} (c_{k-1} - 1) c_0 \cdots c_{i-1}$ for some  $1 \leq i \leq k-1$.
We have
\[
    c_i \cdots c_{k-2} (c_{k-1} - 1) = \infw{a}[i,k-2](c_{k-1}-1) \leq \infw{a}[i,k-1] \leq \infw{a}[0,k-1-i] = c_0 \cdots c_{k-1-i},
\]
where the first and last equalities come from Assumption~\ref{lem:iff-assumption-i}, the first inequality comes from Assumption~\ref{lem:iff-assumption-ii}, and the second inequality comes from the fact that $\infw{a} = w^\omega$ and $w$ is anti-Lyndon.
Therefore, we have two cases.
If $c_i \cdots c_{k-2} (c_{k-1} - 1) < c_0 \cdots c_{k-1-i}$, then
\[
c_i \cdots c_{k-2} (c_{k-1} - 1) c_0 \cdots c_{i-1} < c_0 \cdots c_{k-2} (c_{k-1} - 1),
\]
which is enough in this case.
Otherwise, we get $c_{k-1} - 1 = \infw{a}_{k-1}$, so $c_0 \cdots c_{k-2}(c_{k-1} - 1)$ is an integer power of $w$ by Assumption~\ref{lem:iff-assumption-iii}. Hence it is maximal among its conjugates since $w$ is anti-Lyndon.
\end{proof}

\begin{proposition}
\label{pro:equivalent-conditions-4}
    Let $c$ satisfy~\eqref{eq:WH}. The following assertions are equivalent.
    \begin{enumerate}
        \item \label{pro:equi-cond-1} The word $\infw{u}[0,\length{n+1} - 1)$ is a fractional power of $\word{n}$, i.e., $\length{n+1} - 1 \leq \upperb{n}$, for all $n \geq 0$.
        \item \label{pro:equi-cond-2} We have $\infw{d}^\star[0,n] \leq \infw{a}[0,n]$ for all $n\geq0$.
        \item \label{pro:equi-cond-3} The word $c_0 \cdots c_{k-2}(c_{k-1}-1)$ is maximal among its conjugates.
        \item \label{pro:equi-cond-4} The numeration system $\cS_c$ is greedy.
    \end{enumerate}
\end{proposition}
\begin{proof}
    We prove that~\ref{pro:equi-cond-1} implies~\ref{pro:equi-cond-2} by contraposition.
    Let $n\geq0$ be the smallest integer such that $\infw{d}^\star[0,n] > \infw{a}[0,n]$.
    Note that $n\neq 0$ since $\infw{d}^\star_0 = c_0 = \infw{a}_0$.
    By minimality, we have $\infw{d}^\star[0,n-1] \leq \infw{a}[0,n-1]$, so $\infw{d}^\star[0,n-1] = \infw{a}[0,n-1]$ and $\infw{d}^\star_n > \infw{a}_n$
    Therefore, by the last parts of both Lemma~\ref{L:decomposition of T_n+1 - 1} and Proposition~\ref{P:longest fractional power}, we have $\length{n+1} - 1 > \upperb{n}$.

    We show that~\ref{pro:equi-cond-2} implies~\ref{pro:equi-cond-1} by contraposition. Assume that there exists an integer $n$ such that $\length{n+1} - 1 > \upperb{n}$ and let us show that $\infw{d}^\star[0,n] > \infw{a}[0,n]$. By Proposition~\ref{P:longest fractional power}, $\upperw{n} = \word{n}^{\infw{a}_0} \cdots \word{0}^{\infw{a}_n}$ is a proper prefix of $\infw{u}[0,\length{n+1}-1)$. By Lemma~\ref{L:decomposition of T_n+1 - 1}, $\rep_{\cS_c}(\length{n+1}-1) = \infw{d}^\star[0,n]$, so $\infw{d}^\star_0$ is the largest exponent $e$ such that $\word{n}^e$ is a prefix of $\infw{u}[0,\length{n+1}-1)$. This implies that $\infw{d}^\star_0 \geq \infw{a}_0$. Moreover, if $\infw{a}_0 = \infw{d}^\star_0$, the same argument implies that $\infw{d}^\star_1$ is the largest exponent $e$ such that $\word{n}^{\infw{d}^\star_0}\word{n-1}^e$ is a prefix of $\infw{u}[0,\length{n+1}-1)$.
    In both cases, we have $\infw{d}^\star_0\infw{d}^\star_1 \geq \infw{a}_0\infw{a}_1$. We may  iterate the reasoning to obtain $\infw{d}^\star[0,n] \geq \infw{a}[0,n]$.
    As $\upperw{n}$ is a proper prefix of $\infw{u}[0,\length{n+1}-1)$, the inequality cannot be an equality so we conclude.

    We prove that~\ref{pro:equi-cond-2} implies~\ref{pro:equi-cond-3}.
    Assume that $\infw{d}^\star[0,n] \leq \infw{a}[0,n]$ for all $n \geq 0$. Using the first part of Lemma~\ref{L:inequality of c_0...c_k-2}, this directly implies Items~\ref{lem:iff-assumption-i} and~\ref{lem:iff-assumption-ii} of Lemma~\ref{L:inequality of c_0...c_k-2}. Let us show Item~\ref{lem:iff-assumption-iii} to conclude that $c_0 \cdots c_{k-2}(c_{k-1}-1)$ is maximal among its conjugates by Lemma~\ref{L:inequality of c_0...c_k-2}. Assume that $c_{k-1} -1 = \infw{a}_{k-1}$. Therefore, $c_0 \cdots c_{k-2}(c_{k-1}-1) = w^\ell v$ for some proper prefix $v$ of $w$ and $\ell \geq 1$. Let $u$ be such that $w = vu$. We then have 
    \[
        w^\ell vw = \infw{d}^\star[0, (\ell + 1) |w| + |v|) \leq \infw{a}[0, (\ell + 1) |w| + |v|) = w^\ell vuv.
    \]
    Since $w$ is anti-Lyndon, the only possibility is to have $v = \eps$ and $c_0 \cdots c_{k-2}(c_{k-1}-1)$ is an integer power of $w$.
    This proves Item~\ref{lem:iff-assumption-iii}.

    We show that~\ref{pro:equi-cond-3} implies~\ref{pro:equi-cond-2}.
    By Item~\ref{lem:iff-assumption-ii} of Lemma~\ref{L:inequality of c_0...c_k-2}, we have $c_{k-1} - 1 \le \infw{a}_{k-1}$. 
    If the previous inequality is strict, then the conclusion is direct by Item~\ref{lem:iff-assumption-i} of Lemma~\ref{L:inequality of c_0...c_k-2}.
    Otherwise, by Item~\ref{lem:iff-assumption-iii} of Lemma~\ref{L:inequality of c_0...c_k-2}, $c_0 \cdots c_{k-2}(c_{k-1} - 1)$ is an integer power of $w$ and we conclude that $\infw{a} = \infw{d}^\star$, which is enough.

    Finally, the Assertions~\ref{pro:equi-cond-3} and~\ref{pro:equi-cond-4} are equivalent by Theorem~\ref{T:greedy condition}.
    This ends the proof.
\end{proof}

\begin{remark}
     Examining the proof of~\ref{pro:equi-cond-2} implies~\ref{pro:equi-cond-3} of Proposition~\ref{pro:equivalent-conditions-4}, we observe that it is enough to know that $\infw{d}^\star[0, k-1+|w|] \le \infw{a}[0, k-1+|w|]$, so Assertions~\ref{pro:equi-cond-1} and~\ref{pro:equi-cond-2} may be respectively replaced as follows.
    \begin{enumerate}
        \item[1'.] The word $\infw{u}[0,\length{n+1} - 1)$ is a fractional power of $\word{n}$, i.e., $\length{n+1} - 1 \leq \upperb{n}$, for all $0 \leq n \leq k - 1 + |w|$.
        \item[2'.] We have $\infw{d}^\star[0,n] \leq \infw{a}[0,n]$ for all $0 \leq n \leq k - 1 + |w|$.
    \end{enumerate}
\end{remark}

From Proposition~\ref{P:string attractor imply fractional powers} and Theorem~\ref{T:sa of prefixes}, we directly obtain the following corollary.

\begin{corollary}
    Let $c$ satisfy~\eqref{eq:WH}. Every prefix of $\infw{u}_c$ has a string attractor made of elements of $\{\length{n} : n \geq 0\}$ if and only if one of the four assertions of Proposition~\ref{pro:equivalent-conditions-4} is satisfied.
\end{corollary}

\section{Optimality of the string attractors}
\label{sec: optimality}

So far we were interested in obtaining a precise description (related to a specific numeration system) of a string attractor for each prefix of the infinite word of interest.
In this section, we rather focus on the size of attractors and therefore recall the following concept from~\cite{Schaeffer-Shallit-2020,SA-LATIN22}.
Given an infinite word $\infw{x}$ and any integer $n \geq 1$, we let $s_\infw{x}(n)$ denote the size of a smallest string attractor for the length-$n$ prefix of $\infw{x}$. The function $s_\infw{x} \colon n \mapsto s_\infw{x}(n)$ is called the \emph{string attractor profile function} of $\infw{x}$.
As a consequence of Theorem~\ref{T:sa of prefixes}, we obtain the following.

\begin{corollary}
\label{cor: value of profile function}
    Let $c$ satisfy~\eqref{eq:WH} and assume that one of the assertions of Proposition~\ref{pro:equivalent-conditions-4} holds.
    \begin{enumerate}
        \item For all $n\ge 0$ and all $m \in [\lowerb{n},\upperb{n}]$, $s_\infw{u}(m) = \card(\Gamma_n)$.
        \item For all large enough $m$, $k \leq s_\infw{u}(m) \leq k+1$.
    \end{enumerate}
\end{corollary}

\begin{proof}
    Using a simple induction, one can check that the positions in $\Gamma_n$ all correspond to different letters in $\infw{u}$. As any string attractor must contain at least one position per letter, we deduce the first claim from Theorem~\ref{T:sa of prefixes}.
    The second part similarly follows from Theorem~\ref{T:sa of prefixes} since, for all large enough $n$, $\Gamma_n$ is of size $k$, and the intervals $[\length{n},\upperb{n}]$ cover $\N \setminus \{0\}$ by Proposition~\ref{pro:equivalent-conditions-4}.
\end{proof}

For some parameters $c$, the intervals $[\lowerb{n},\upperb{n}]$ also cover $\N \setminus \{0\}$, implying that a string attractor of minimal size can always be given by some $\Gamma_n$. This is for example the case of $c = 211$. We characterize these parameters $c$ in the following result.
Recall the definition of the sequence $(\lowerb{n})_{n\geq 0}$ given in Equation~\eqref{eq:def-of-Pn}:
\begin{equation*}
        \lowerb{n} =
    \begin{cases}
    \length{n}, & \text{if } 0\le n \leq k-1;\\
    \length{n} + \length{n-k+1} - \length{n-k} - 1, & \text{if } n \geq k.
    \end{cases}
\end{equation*}

\begin{proposition}\label{P:condition to have size k}
    Let $c$ satisfy~\eqref{eq:WH} and let $w$ denote the longest anti-Lyndon prefix of $c_0 \cdots c_{k-2}$.
    Moreover assume that one of the assertions of Proposition~\ref{pro:equivalent-conditions-4} holds. The inequality $\lowerb{n} - 1 \leq \upperb{n-1}$ holds for all $n$ if and only if the following conditions are satisfied:
    \begin{enumerate}
        \item we have $c_{k-1} = 1$;
        \item the word $c_0 \cdots c_{k-2}$ is an integer power of $w$.
    \end{enumerate}
\end{proposition}

\begin{proof}
    Observe that, for all $n \leq k-1$, Equation~\eqref{eq: recurrence for U_n} gives
    \[
        \lowerb{n} - 1 = \length{n} - 1 = c_0 \length{n-1} + \dots + c_{n} \length{0}.
    \]
    On the other hand, Lemma~\ref{L:inequality of c_0...c_k-2} implies that $c_0 \cdots c_{k-2} = \infw{a}[0,k-2]$, so
    \[
        \upperb{n-1} = \infw{a}_0 \length{n-1} + \dots + \infw{a}_n \length{0} = \lowerb{n} - 1
    \]
    by Proposition~\ref{P:longest fractional power}. This shows that the inequality of the statement always holds for all $n \leq k-1$. Let us now show that it is also satisfied for $n \geq k$ if and only if the two conditions of the statement are satisfied.

    For $n \geq k$, we have
    \begin{align*}
        \lowerb{n} - 1
        &= \length{n} + \length{n-k+1} - \length{n-k} - 2\\
        &= c_0 \length{n-1} + \dots + c_{k-2} \length{n-k+1} + (c_{k-1}-1)\length{n-k} + \length{n-k+1} - 2,
    \end{align*}
    where the second equality follows from Equation~\eqref{eq: recurrence for U_n}, and
    \begin{align*}
        \upperb{n-1}
        &= \infw{a}_0\length{n-1} + \dots + \infw{a}_{k-2} \length{n-k+1} + \infw{a}_{k-1} \length{n-k} + \dots + \infw{a}_{n-1} \length{0}.
    \end{align*}
    Therefore, by Lemma~\ref{L:inequality of c_0...c_k-2},
    \begin{equation}\label{Eq:equivalent inequality}
        \lowerb{n} - 1 \leq \upperb{n-1} \iff (c_{k-1} - 1) \length{n-k} + \length{n-k+1} - 2 \leq \infw{a}_{k-1}\length{n-k} + \dots + \infw{a}_{n-1}\length{0}.
    \end{equation}
    
    Now, if the two conditions of the statement are fulfilled, i.e., if $c_{k-1} = 1$ and $c_0 \cdots c_{k-2}$ is an integer power of $w$, then $\infw{a}_{k-1+\ell} = \infw{a}_{\ell}$ for all $\ell \geq 0$, so
    \[
        \lowerb{n} - 1 \leq \upperb{n-1} \iff \length{n-k+1} - 2 \leq \upperb{n-k}.
    \]
    Since $\length{n-k+1} - 2 \leq \lowerb{n-k+1} - 1$, we can then conclude by induction that the inequality of the statement is satisfied for all $n\ge 0$.

    Let us now prove the converse by a detailed case-analysis, and assume first by contradiction that $c_{k-1} \ne 1$. Since $c_{k-1} \geq 1$, this in fact means that $c_{k-1} \geq 2$. For $n = k$, the right-hand side of~\eqref{Eq:equivalent inequality} becomes
    \[
        (c_{k-1} - 1) \length{0} + \length{1} - 2 \leq \infw{a}_{k-1} \length{0},
    \]
    which gives, since $\length{0} = 1$ and $\length{1} = c_0 + 1$,
    \[
        c_0 \leq c_{k-1} - 2 + c_0 \leq \infw{a}_{k-1}.
    \]
    Since $\infw{a}$ is the infinite concatenation of an anti-Lyndon word, we have $\infw{a}_{k-1} \leq \infw{a}_0 = c_0$. Therefore, we deduce that $\infw{a}_{k-1} = c_0$ and $c_{k-1} = 2$.
    Now, if $n = k+1$, the right-hand side of~\eqref{Eq:equivalent inequality} becomes
    \[
    (c_{k-1}-1) \length{1} + \length{2} - 2 \leq \infw{a}_{k-1}\length{1} + \infw{a}_k\length{0}
        \iff c_0 + 1 + (\length{2} - c_0 \length{1}) - 2 \leq \infw{a}_k.
    \]
    This leads us to consider two cases depending on $k$.
    If $k=2$, we get $\length{2} - c_0 \length{1} = c_1 = 2$ so $c_0 + 1 + 2 - 2 \le \infw{a}_2$, but as $\infw{a}_2 \leq \infw{a}_0 = c_0$, this is impossible.
    If $k\ge 3$, then $\length{2} - c_0 \length{1} = c_1 + 1$, however since $\infw{a}_{k-1}\infw{a}_k \leq \infw{a}_0\infw{a}_1 = c_0c_1$ and $\infw{a}_{k-1} = c_0$ as established above, we must have $c_0 + 1 + c_1 + 1 - 2 \le \infw{a}_k \le c_1$ which is also impossible.

    This shows that, if $\lowerb{n} - 1 \leq \upperb{n-1}$ for all $n \geq0$, then $c_{k-1} = 1$. Assume now by contradiction that $c_0 \cdots c_{k-2}$ is not an integer power of $w$. Therefore, by Lemma~\ref{L:inequality of c_0...c_k-2}, there exist $\ell \geq 1$ and a proper non-empty prefix $x$ of $w$ such that such that $c_0 \cdots c_{k-2} = w^\ell x$. Let us denote $w = xy$ for some (non-empty) word $y$.

    For $n = k + |w|$, the right-hand side of~\eqref{Eq:equivalent inequality} becomes
    \begin{align}\label{Eq:failed inequality}
        \nonumber
        \length{|w|+1} - 2 &\leq \infw{a}_{k-1}\length{|w|} + \dots + \infw{a}_{k-1+|w|} \length{0}\\
        \iff c_0\length{|w|} + \dots + (c_{|w|} - 1) \length{0} &\leq \infw{a}_{k-1}\length{|w|} + \dots + \infw{a}_{k-1+|w|} \length{0},
    \end{align}
    where we used Equation~\eqref{eq: recurrence for U_n} to develop $\length{|w|+1}$ (recall that $|w| + 1 \leq \ell |w| + |x| = k-1$).
    Notice that $\infw{a}[k-1,+\infty) = yw^\omega$, so $\infw{a}[k-1,k-1+|w|] = yxy_0$. On the other hand, $c_0 \cdots c_{|w|-1}(c_{|w|} - 1) = w(w_0 - 1)$ (recall that $w_0 \geq 1$).
    As $w$ is anti-Lyndon, we have $w > yx$ so $w(w_0-1) > yxy_0$. To obtain a contradiction, let us go back to numeration systems.

    For all $i \leq |w| + 1\le k+1$, the length-$i$ suffix $v$ of $yxy_0$ (resp., $w(w_0-1)$) is (resp., is smaller than) a factor of $\infw{a}$, so, since $w$ is anti-Lyndon, $v \leq \infw{a}[0,i) = \infw{d}^\star[0,i)$ by Lemma~\ref{L:inequality of c_0...c_k-2}. Therefore, by Lemma~\ref{L:greedyIFFlanguage} and Proposition~\ref{pro:equivalent-conditions-4}, $yxy_0$ (resp., $w(w_0-1)$) is in the numeration language of $\cS_c$. By Theorem~\ref{T:automata construction of numeration systems}, $\cS_c$ respects the genealogical order, and by Remark~\ref{R:valuation of numeration systems}, the valuation is given by the sequence $(\length{n})_{n \geq 0}$. Therefore, the word inequality $w(w_0-1) > yxy_0$ implies the (integer) inequality
    \[
    c_0\length{|w|} + \dots + c_{|w| - 1} + (c_{|w|}-1) \length{0} > \infw{a}_{k-1}\length{|w|} + \dots + \infw{a}_{k-1+|w|} \length{0}, 
   \]
    which contradicts Inequality~\eqref{Eq:failed inequality}. This ends the proof that, if $\lowerb{n} - 1 \leq \upperb{n-1}$ for all $n$, then $c_{k-1} = 1$ and $c_0 \cdots c_{k-2}$ is an integer power of $w$.
\end{proof}

We immediately obtain the next result.

\begin{corollary}
    Let $c$ satisfy~\eqref{eq:WH}. If $c = w^\ell 1$ for some $\ell \in \N \setminus \{0\}$ and some anti-Lyndon word $w$, then
    \[
        s_\infw{u}(m) = 
        \begin{cases}
            i+1, & \text{ if $m \in [\length{i}, \length{i+1})$ with $0 \leq i \leq k-2$};\\
            k, & \text{ if $m \geq \length{k-1}$}.
        \end{cases}
    \]
\end{corollary}

\begin{remark}
    We note that the conditions of Proposition~\ref{P:condition to have size k} are precisely those of~\cite[Theorem 1.1.]{BMP2007}, which characterizes the words $\infw{u}_c$ for which there exists a simple Parry number $\beta$ such that $d_\beta(1)=c0^\omega$ and the factor complexity function of $\infw{u}_c$ is affine.
\end{remark}

However, Proposition~\ref{P:condition to have size k} does not ban the existence of other parameters $c$ for which a string attractor of minimal size can always be given by some $\Gamma_n$. This follows from the next remark.

\begin{remark}\label{R:optimality of bounds}
The proof of Proposition~\ref{P:string attractor imply fractional powers} shows that the factor $\infw{u}[\length{n},\upperb{n}]$ does not appear before. Therefore $\Gamma_n$ cannot be a string attractor of $\infw{u}[0,m)$ if $m \geq \upperb{n} + 1$; in other words the upper bound $\upperb{n}$ is tight in Theorem~\ref{T:sa of prefixes}.
However, the lower bound $\lowerb{n}$ is not necessarily tight. For example, if $c=23$, then $\Gamma_2 = \{3,9\}$ is a string attractor of the length-$9$ prefix  $\infw{u}[0,9) = 00\underline{1}00100\underline{0}$, while $P_2 = 10$.
This is also the case for the $k$-bonacci morphisms ($c = 1^k$) where better bounds are provided in~\cite{CGRRSS23}.
\end{remark}

On the other hand, there exist parameters $c$ satisfying the conditions of Proposition~\ref{pro:equivalent-conditions-4} but for which $\Gamma_n$ is sometimes not sufficient. The simplest such example is the \emph{period-doubling word} corresponding to $c = 12$. Indeed, the length-$8$ prefix is given by $01000101$ and the first few elements of the sequence $(\length{n})_{n \geq 0}$ are $1,2,4,8$. One then easily checks that none of $\Gamma_1 = \{1,2\}$, $\Gamma_2 = \{2,4\}$, and $\Gamma_3 = \{4,8\}$ is a string attractor. However, this word admits the size-$2$ string attractors $\{3,6\}$ and $\{4,6\}$ for example. In fact, Schaeffer and Shallit~\cite{Schaeffer-Shallit-2020} proved that, if $\infw{u}$ is the period-doubling word,
\[
    s_\infw{u}(n) =
    \begin{cases}
        1, & \text{ if $n = 1$};\\
        2, & \text{ if $n \geq 2$}.
    \end{cases}
\]

Based on this observation and our experiments, we state the following conjecture.

\begin{conjecture}
    Let $c$ satisfy~\eqref{eq:WH} and assume that one of the assertions of Proposition~\ref{pro:equivalent-conditions-4} holds. Then,
    \[
        s_\infw{u}(m) = 
        \begin{cases}
            i+1, & \text{ if $m \in [\length{i}, \length{i+1})$ with $0 \leq i \leq k-2$};\\
            k, & \text{ if $m \geq \length{k-1}$}.
        \end{cases}
    \]
\end{conjecture}

Observe that, using the first part of Corollary~\ref{cor: value of profile function}, this equality is known to be true for all $m \leq \length{k} - 1$ and for infinitely many values of $m$.

\bmhead{Acknowledgments}

We warmly thank \'E. Charlier, S. Kreczman, and M. Rigo for useful discussions on numeration systems, especially the last two indicated~\cite{Hollander} and~\cite{Dumont-Thomas-1989} respectively.

Giuseppe Romana is partly supported by MUR PRIN project no. 2022YRB97K - 'PINC' (Pangenome INformatiCs: From Theory to Applications).
Manon Stipulanti is an FNRS Research Associate supported by the Research grant 1.C.104.24F.

\bibliography{biblio.bib}

\end{document}